\documentclass[10pt,letterpaper]{article}

\usepackage[accepted]{tmlr}
\def\code{%
	\textit{https://github.com/pnkraemer/code-numerically-robust-fixedpoint-smoother}
}
\def\acks{%
\subsection*{Acknowledgements}
This work was supported by a research grant (42062) from VILLUM FONDEN. The work was partly funded by the Novo Nordisk Foundation through the Center for Basic Machine Learning Research in Life Science (NNF20OC0062606). This project received funding from the European Research Council (ERC) under the European Union’s Horizon programme (grant agreement 101125993).
}

\usepackage{amsthm}

\usepackage[colorlinks, citecolor=blue!50!black, urlcolor=black, linkcolor=black]{hyperref}
\usepackage{url}

\usepackage{tikz}
\usetikzlibrary{positioning, fit, bayesnet}
\usetikzlibrary{arrows.meta, arrows, backgrounds}
\tikzset{>=stealth'}

\usepackage{colortbl}
\newcommand{\winner}{\cellcolor{gray!20!white}\mathbf}

\usepackage{blindtext}

\usepackage{mwe}

\usepackage{microtype}

\usepackage{amsmath}
\newcommand{\oderiv}[1]{\frac{d^{#1}}{dt^{#1}}}

\usepackage{amssymb}
\newcommand{\Rbb}{\mathbb{R}}

\usepackage{cleveref}

\usepackage{mathtools}

\newtheorem{algo}{Algorithm}
\newtheorem{proposition}{Proposition}
\newtheorem*{contribution}{Contribution}
\crefname{algo}{algorithm}{algorithms}

\creflabelformat{equation}{#2#1#3}

\usepackage{algorithm}
\usepackage{algpseudocode}

\newcommand{\diff}{\,\text{d}}

\newcommand{\Chol}[1]{L_{#1}}

\usepackage{title_and_author}

\usepackage{booktabs}

\usepackage{amssymb}%
\usepackage{pifont}%
\newcommand{\cmark}{ { \color{green!65!black}\ding{51} } }%
\newcommand{\xmark}{ { \color{red} \ding{55} } }%

\renewcommand{\paragraph}[1]{ { \sffamily {\textbf{#1} } } }

\usepackage{wrapfig}

\usepackage{xfrac}

\def\month{01}  %
\def\year{2025} %
\def\openreview{\url{https://openreview.net/forum?id=LVQ8BEL5n3}} %

\begin{document}
\maketitle

\begin{abstract}
Practical implementations of Gaussian smoothing algorithms have received a great deal of attention in the last 60 years.
However, almost all work focuses on estimating complete time series (``fixed-\emph{interval} smoothing'', $\mathcal{O}(K)$ memory) through variations of the Rauch--Tung--Striebel smoother, rarely on estimating the initial states (``fixed-\emph{point} smoothing'', $\mathcal{O}(1)$ memory). 
Since fixed-point smoothing is a crucial component of algorithms for dynamical systems with unknown initial conditions, we close this gap by introducing a new formulation of a Gaussian fixed-point smoother.
In contrast to prior approaches, our perspective admits a numerically robust Cholesky-based form (without downdates) and avoids state augmentation, which would needlessly inflate the state-space model and reduce the numerical practicality of any fixed-point smoother code.
The experiments demonstrate how a JAX implementation of our algorithm matches the runtime of the fastest methods and the robustness of the most robust techniques while existing implementations must always sacrifice one for the other.

Code: \code{}
\end{abstract}

\section{Introduction}
Linear Gaussian state-space models and Bayesian filtering and smoothing enjoy numerous applications in areas like tracking, navigation, or control \citep{grewal2014kalman,sarkka2023bayesian}.
They also serve as the computational backbone of contemporary machine learning methods that revolve around time-series, streaming data, or sequence modelling; for example, message passing, Gaussian processes, and probabilistic numerics \citep{grewal2014kalman,sarkka2019applied,murphy2023probabilistic,hennig2022probabilistic}.
Filtering and smoothing have also been used for constructing and training neural networks \citep{singhal1988training,gu2023mamba,chang2023lowrank} or continual learning \citep{sliwa2024efficient}.
All of these applications require fast and robust algorithms for state-space models; we propose a new one in this paper.

Recent developments in Bayesian smoothing focus on subjects like linearisation \citep[e.g.][]{garcia2016iterated}, temporal parallelisation \citep[e.g.][]{sarkka2020temporal}, or numerically robust implementations \citep[e.g.][]{yaghoobi2022parallel}.
However, they all exclusively focus on fixed-\emph{interval} smoothing, which targets the full time-series $p(x_{0:K} \mid y_{1:K})$ in $\mathcal{O}(K)$ memory, never on fixed-\emph{point} smoothing, which only yields the initial value $p(x_0 \mid y_{1:K})$ but does so in $\mathcal{O}(1)$ memory.
Even though fixed-point smoothing is a crucial component of, for example, estimating past locations of a spacecraft \citep{meditch1969stochastic}, it has yet to receive much attention in the literature on state estimation in dynamical systems. 
The experiments in \Cref{section-experiments} demonstrate fixed-point smoothing applications in probabilistic numerics and in a tracking problem.
We anticipate that probabilistic numerical solvers for differential equations will especially benefit from improving the practicality of fixed-point smoothers.
The reason is that these algorithms closely connect to filtering and smoothing algorithms \citep{schober2019probabilistic}, especially the numerically robust kind \citep{kramer2024stable}, and that unknown initial conditions are typical for parameter estimation problems in differential equations and scientific machine learning \citep{rackauckas2020universal}.
\Cref{section-experiment-robustness-bvp} revisits probabilistic numerics as a prime application for fixed-point smoothing.

\paragraph{Contributions}
We introduce an implementation of numerically robust fixed-point smoothing.
Our approach avoids the typical construction of fixed-point smoothers via state-augmentation \citep{biswas1972approach,smith1982robustness}, which increases the dimension of the state-space model and thus makes estimation needlessly expensive. And unlike previous work on fixed-point smoothing \citep{meditch1967optimal,sarkka2010gaussian,rauch1963solutions,meditch1967orthogonal,meditch1976initial,meditch1973estimation,nishimura1969new}, our perspective is not tied to any particular parametrisation of Gaussian variables.
Instead, our proposed algorithm enjoys numerical robustness and compatibility with data streams, while maintaining minimal complexity.
There exist tools for each of those desiderata, but our algorithm is the first to deliver them all at once.
We demonstrate the algorithm's efficiency and robustness on a sequence of test problems, including a probabilistic numerical method for boundary value problems \citep{kramer2021linear}, and show how to use fixed-point smoothing for parameter estimation in Gaussian state-space models.

\paragraph{Notation}
Enumerated sets are abbreviated with subscripts, for example $x_{0:K} \coloneqq \{x_0, ..., x_K\}$ and ${ y_{1:K} \coloneqq \{y_1, ..., y_K\} }$.
For sequential conditioning of Gaussian variables (like in the Kalman filter equations), we indicate the most recently visited data points with subscripts in the parameter vectors, for example, $p(x_k \mid y_{1:k-1}) = \mathcal{N}(m_{k \mid k-1}, C_{k \mid k-1})$ or $\mathcal{N}(m_{K \mid K}, C_{K \mid K}) = p(x_K \mid y_{1:K})$.
$I_n$ is the identity matrix with $n$ rows and columns.
All covariance matrices shall be symmetric and positive semidefinite.
Like existing work on numerically robust state-space model algorithms \citep[e.g.][]{grewal2014kalman,kramer2024implementing},
we define the \emph{generalised Cholesky factor} $\Chol{\Sigma}$ of a covariance matrix $\Sigma$ as any matrix that satisfies $\Sigma = \Chol{\Sigma} (\Chol{\Sigma})^\top$.
This definition includes the ``true'' Cholesky factor if $\Sigma$ is positive definite but applies to semidefinite $\Sigma$; for instance, the zero-matrix is the generalised Cholesky factor of the zero-matrix (which does not admit a Cholesky decomposition).
There are numerous ways of parametrising multivariate Gaussian distributions, but we exclusively focus on two kinds. Like \citet{yaghoobi2022parallel}, we distinguish:
\begin{itemize}
	\item \emph{Covariance-based parametrisations:} Parametrise multivariate Gaussian distributions with means and covariance matrices. Covariance-based parametrisations are the standard approach.
	\item \emph{Cholesky-based parametrisations:} Parametrise multivariate Gaussian distributions with means and generalised Cholesky factors of covariance matrices instead of covariance matrices. 
	Manipulating Gaussian variables in Cholesky-based parametrisations replaces the addition and subtraction of covariance matrices with QR decompositions, which improves the numerical robustness at the cost of a slightly increased runtime; it leads to methods like the square-root Kalman filter \citep{bennett1965triangular}.
\end{itemize}
Other forms, such as the information or canonical form of a multivariate Gaussian distribution \citep{murphy2022probabilistic}, are not directly relevant to this work.

\section{Problem statement: Fixed-point smoothing}
\label{section-background-on-filtering-and-smoothing}

\subsection{Background on filtering and smoothing}
\label{section-linear-gaussian-ssms}

\paragraph{Linear Gaussian state-space models}
This work only discusses linear, discrete state-space models with additive Gaussian noise because such models are the starting point for Bayesian filtering and smoothing. 
Nonlinear extensions are future work.
For integers $d$ and $D$, let $A_1, ..., A_K \in \Rbb^{D \times D}$ and $H_1, ..., H_K \in \Rbb^{d \times D}$ be linear operators and $C_{0 \mid 0}, B_1, ..., B_K \in \Rbb^{D \times D}$ and $R_1, ..., R_K \in \Rbb^{d \times d}$ be covariance matrices.
Introduce a vector $m_{0 \mid 0} \in \Rbb^D$.
Assume that observations $y_{1:K}$ are available according to the state-space model (and potentially as a data stream)
\begin{align}
\label{equation-state-space-model}
x_0 = \theta,
\quad
x_{k} = A_k x_{k-1} + b_k,
\quad
y_k = H_{k} x_{k} + r_{k},
\quad
k=1, ..., K,
\end{align}
with pairwise independent Gaussian random variables
\begin{align}
\label{equation-state-space-model-noise}
\theta  \sim \mathcal{N}(m_{0 \mid 0}, C_{0 \mid 0}),
\quad
b_k \sim \mathcal{N}(0, B_k),
\quad
r_{k} \sim \mathcal{N}(0, R_{k}),
\quad
k=1, ..., K.
\end{align}
In order to simplify the notation, \Cref{equation-state-space-model,equation-state-space-model-noise} assume that the Gaussian variables $b_{1:K}$ and $r_{1:K}$ have a zero mean and that there is no observation of the initial state $x_0 = \theta$.
The experiments in \Cref{section-experiments} demonstrate successful fixed-point smoothing even when violating those two assumptions.
Estimating $x_{1:K}$ from observations $y_{1:K}$ is a standard setup for filtering and smoothing \citep{sarkka2023bayesian}.

\paragraph{Kalman filtering}
Different estimators target different conditional distributions
(\Cref{table-state-estimation-approaches}).
\begin{table}[t]
\caption{
\emph{Estimation in Gaussian state-space models.}
Other estimation tasks exist, for example, {fixed-lag smoothing}. However, this article focuses on fixed-point smoothing and its relationship with filtering and fixed-interval smoothing.
``RTS smoother'': ``Rauch--Tung--Striebel smoother''.
}
\label{table-state-estimation-approaches}
\begin{center}
\begin{tabular}{ l l }
\toprule
Task & Target 
\\
\midrule
Filtering (e.g. Kalman filter) & Real-time updates $\{p(x_{k} \mid y_{1:k})\}_{k=1}^K$
\\
Fixed-interval smoothing (e.g. RTS smoother) & Full time series $p(x_{0:K} \mid y_{1:K})$
\\
Fixed-point smoothing & Initial state $p(x_{0} \mid y_{1:K})$
\\
\bottomrule
\end{tabular}
\end{center}
\end{table}
For example, the \emph{Kalman filter} \citep{kalman1960new} computes $\{p(x_k \mid y_{1:k})\}_{k=1}^K$ by initialising $p(x_0 \mid y_{1:0}) = p(x_0)$ and alternating
\begin{subequations}
\label{equation-filtering-pass}
\begin{align}
\text{Prediction:}&
&
p(x_{k-1} \mid y_{1:k-1}) 
&\longmapsto p(x_{k} \mid y_{1:k-1})
&
k&=1,...,K
\label{equation-filtering-pass-prediction}
\\
\text{Update:}&
&
p(x_{k} \mid y_{1:k-1}) 
&\longmapsto p(x_{k} \mid y_{1:k})
&
k&=1,...,K.
\end{align}
\end{subequations}
Since all operations are linear and all distributions Gaussian, the recursions are available in closed form.
Detailed iterations are in \Cref{appendix-section-kalman-filter}.
An essential extension of the Kalman filter is the \emph{square-root Kalman filter} \citep{bennett1965triangular,andrews1968square} (see also \citet{grewal2014kalman}), which yields the identical distributions as the Kalman filter but manipulates Cholesky factors of covariance matrices instead of covariance matrices.
The advantage of the Cholesky-based implementation of the Kalman filter over the covariance-based version is that covariance matrices are guaranteed to remain symmetric and positive semidefinite, whereas accumulated round-off errors can make the covariance-based Kalman filter break down.
In this work, we introduce Cholesky-based implementations for fixed-point smoothing, among other things.
Both the Kalman filter and its Cholesky-based extension cost $\mathcal{O}(K D^3)$ runtime and $\mathcal{O}(D^2)$ memory, assuming $D \geq d$ (otherwise, exchange $D$ for $d$).
The Cholesky-based filter is slightly more expensive because it uses QR decompositions instead of matrix multiplication; \Cref{appendix-section-kalman-filter} contrasts Cholesky-based and covariance-based  Kalman filtering.

\paragraph{Fixed-interval smoothing}
While the filter computes the terminal state $p(x_K \mid y_{1:K})$, the \emph{fixed-interval smoother} targets the entire trajectory $p(x_{0:K} \mid y_{1:K})$.
For linear Gaussian state-space models, fixed-interval smoothing is implemented by the (fixed-interval) Rauch--Tung--Striebel smoother \citep{rauch1965maximum}:
Relying on independent noise in \Cref{equation-state-space-model-noise}, the Rauch--Tung--Striebel smoother factorises the conditional distribution backwards in time according to (interpret $y_{1:0} = \emptyset$ for notational brevity), 
\begin{align}
\label{equation-smoothing-backward-factorisation}
p(x_{0:K} \mid y_{1:K}) = p(x_K \mid y_{1:K}) \prod_{k=1}^{K}p(x_{k-1} \mid x_{k}, y_{1:k-1}).
\end{align}
The first term, $p(x_K \mid y_{1:K})$, is the filtering distribution at the final state and usually computed with a Kalman filter.
The remaining terms, $p(x_{k-1} \mid x_{k}, y_{1:k-1})$, can be assembled with the prediction step in the filtering pass (\Cref{equation-filtering-pass-prediction}), which preserves the Kalman filter's $\mathcal{O}(K D^3)$ runtime complexity but increases the memory consumption from $\mathcal{O}(D^2)$ to $\mathcal{O}(K D^2)$.
Like the Kalman filter, the Rauch--Tung--Striebel smoother allows covariance- and Cholesky-based parametrisations \citep{park1995square} in roughly the same complexity. 
In practice, Cholesky-based arithmetic costs slightly more than covariance-based arithmetic but enjoys an increase in numerical robustness.
\Cref{appendix-section-rts-smoother} contrasts both implementations.

\paragraph{Towards fixed-point smoothing}
Finally, we turn to the \emph{fixed-point smoothing} problem: computing the conditional distribution of the initial state $p(x_0 \mid y_{1:K})$ conditioned on all observations (\Cref{figure-fixed-point-smoothing-graph}).
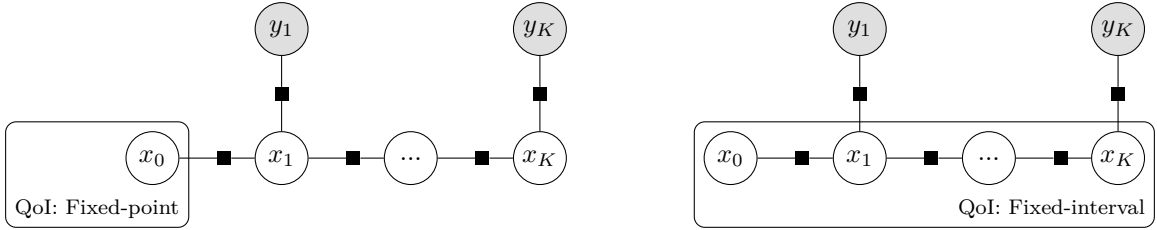
\begin{figure}[t]
\begin{center}
\begin{tabular}{ l c c c r }
\begin{tikzpicture}
\node[latent] (x1) {$x_1$};
\node[latent, left=1.0cm of x1] (x0) {$x_0$};
\node[latent, right=1.0cm of x1] (dots1) {...};
\node[latent, right=1.0cm of dots1] (xK) {$x_K$};

\node[obs, above=of x1] (y1) {$y_1$};
\node[obs, above=of xK] (yK) {$y_K$};

\factor[above=of x1]     {f1}     {} {} {} ; %
\factor[above=of xK]     {fK}     {} {} {} ; %

\factor[right=0.5cm of x0]     {f01}     {} {} {} ; %
\factor[right=0.5cm of x1]     {f1k}     {} {} {} ; %
\factor[right=0.5cm of dots1]     {f1K}     {} {} {} ; %

\factoredge[-] {x1} {f1} {y1};
\factoredge[-] {xK} {fK} {yK};

\factoredge[-] {x0} {f01} {x1};
\factoredge[-] {x1} {f1k} {dots1};
\factoredge[-] {dots1} {f1K} {xK};

\plate {target} {(x0)} {QoI: Fixed-point};
\end{tikzpicture}
&
&
&
&
\begin{tikzpicture}
\node[latent] (x1) {$x_1$};
\node[latent, left=1.0cm of x1] (x0) {$x_0$};
\node[latent, right=1.0cm of x1] (dots1) {...};
\node[latent, right=1.0cm of dots1] (xK) {$x_K$};

\node[obs, above=of x1] (y1) {$y_1$};
\node[obs, above=of xK] (yK) {$y_K$};

\factor[above=of x1]     {f1}     {} {} {} ; %
\factor[above=of xK]     {fK}     {} {} {} ; %

\factor[right=0.5cm of x0]     {f01}     {} {} {} ; %
\factor[right=0.5cm of x1]     {f1k}     {} {} {} ; %
\factor[right=0.5cm of dots1]     {f1K}     {} {} {} ; %

\factoredge[-] {x1} {f1} {y1};
\factoredge[-] {xK} {fK} {yK};

\factoredge[-] {x0} {f01} {x1};
\factoredge[-] {x1} {f1k} {dots1};
\factoredge[-] {dots1} {f1K} {xK};

\plate {target} {(x0) (xK)} {QoI: Fixed-interval};
\end{tikzpicture}
\end{tabular}
\end{center}
\caption{Fixed-point (left) versus fixed-interval smoothing problem (right) as factor graphs. The shaded variables are observed. ``QoI'': ``Quantity of interest''.}
\label{figure-fixed-point-smoothing-graph}
\end{figure}
The central difficulty for fixed-point smoothing is that we ask for $\mathcal{O}(K)$ algorithms even though all observations $y_{1:K}$ are in the ``future'' of $x_0$, which lacks an immediate sequential formulation.
That said, one algorithm for sequentially solving the fixed-point smoothing problem involves Rauch--Tung--Striebel smoothing:
The fixed-point smoothing problem could be solved by assembling $p(x_{0:K} \mid y_{1:K})$ with the Rauch--Tung--Striebel smoother, which yields a parametrisation of $p(x_0 \mid y_{1:K})$ in closed form.
Unfortunately, this solution is not very efficient: 
the Rauch--Tung--Striebel smoother stores parametrisations of all conditional distributions $\{p(x_{k-1} \mid x_k, y_{1:k-1})\}_{k=1}^K$, which is problematic for long time-series because it consumes $\mathcal{O}(K D^2)$ memory. 
In contrast, state-of-the-art {fixed-point smoothers} require only $\mathcal{O}(D^2)$ memory.

\subsection{Existing approaches to fixed-point smoothing}
\label{section-existing-approaches-to-fixed-point-smoothing}

The state of the art in fixed-point smoothing does not use fixed-interval smoothers to compute $p(x_0 \mid x_{1:K})$.
There are two more efficient approaches:
Either we solve the fixed-point smoothing problem by computing the solution to a specific high-dimensional filtering problem \citep[{``fixed-point smoothing via state-augmented filtering'' below}]{biswas1972approach}, or we derive recursions for how $p(x_0 \mid y_{1:k})$ evolves with $k=1, ..., K$ \citep[``fixed-point recursions'' below]{meditch1967orthogonal,meditch1967optimal,meditch1969stochastic}.
Both solutions have advantages and disadvantages. We discuss both before explaining how our algorithm fixes their shortcomings.

\paragraph{Fixed-point smoothing via state-augmented filtering}
Fixed-point smoothing can be implemented with a Kalman filter, more precisely, by running a filter on the state-space model \citep{biswas1972approach}
\begin{align}
\label{equation-state-space-model-augmented}
\begin{pmatrix}
x_{k} \\
x_0
\end{pmatrix}
= 
\begin{pmatrix}
A_k & 0 \\
0 & I_D
\end{pmatrix}
\begin{pmatrix}
x_{k-1} \\
x_0
\end{pmatrix}
+
\begin{pmatrix}
I_D \\
0
\end{pmatrix}
b_k,
\quad
y_{k}
= 
\begin{pmatrix}
H_k & 0
\end{pmatrix}
\begin{pmatrix}
x_{k} \\
x_0
\end{pmatrix}
+
r_k,
\quad
k=1, ..., K
\end{align}
with initial condition
\begin{align}
\label{equation-state-space-model-augmented-initial-condition}
\begin{pmatrix}
x_0 \\
x_0 
\end{pmatrix}
\sim
N\left(
\begin{pmatrix}
m_{0 \mid 0} \\
m_{0 \mid 0} 
\end{pmatrix},
\begin{pmatrix}
C_{0 \mid 0} & C_{0 \mid 0} \\
C_{0 \mid 0} & C_{0 \mid 0}
\end{pmatrix}
\right)
=
N\left(
\begin{pmatrix}
m_{0 \mid 0} \\
m_{0 \mid 0} 
\end{pmatrix},
\begin{pmatrix}
\Chol{C_{0 \mid 0}} & 0 \\
\Chol{C_{0 \mid 0}} & 0
\end{pmatrix}
\begin{pmatrix}
\Chol{C_{0 \mid 0}} & 0 \\
\Chol{C_{0 \mid 0}} & 0
\end{pmatrix}^\top
\right)
\end{align}
The covariance of the initial condition is rank-deficient. However, the generalised Cholesky factor remains well-defined according to \Cref{equation-state-space-model-augmented-initial-condition}. 
The difference between the state-space models in \Cref{equation-state-space-model-augmented} and \Cref{equation-state-space-model} is that \Cref{equation-state-space-model-augmented} tracks the augmented state $(x_k, x_0)$ instead of $x_k$.
As a result, the Kalman filter applied to \Cref{equation-state-space-model-augmented,equation-state-space-model-augmented-initial-condition} yields parametrisations of the conditional distributions $
\big\{p(x_k, x_0 \mid y_{1:k})\big\}_{k=1}^K$
in $\mathcal{O}(K (2D)^3)$ runtime and $\mathcal{O}((2D)^2)$ memory.
The factor ``$2D$'' stems from doubling the state-space dimension.
The initial condition $p(x_0 \mid y_{1:K})$ emerges from the augmented state $p(x_K, x_0 \mid y_{1:K})$ in closed form.
Relying on the Kalman filter formulation inherits all computational benefits of Kalman filters, most importantly, the Kalman filter's Cholesky-based form with its attractive numerical robustness.
Nonetheless, doubling of the state-space dimension is unfortunate because it increases the runtime roughly by factor $8$ (because $(2D)^3 = 8 D^3$) and the memory by factor $4$ (because $(2D)^2 = 4 D^2$).
However, as long as one commits to covariance-based arithmetic, this increase can be avoided:

\paragraph{Fixed-point recursions}
\label{section-fixed-point smoother-reduced}
Suppose we temporarily ignore Cholesky-based forms and commit to covariance-based parametrisations of Gaussian distributions. 
In that case, we can improve the efficiency of the fixed-point smoother.
It turns out that during the forward pass, the mean and covariance of $p(x_0 \mid y_{1:k})$ follow a recursion that can be run parallel to the Kalman filter pass \citep{meditch1969stochastic,sarkka2010gaussian,meditch1967orthogonal}.
\citet{sarkka2010gaussian} explain how to implement this recursion:
start by initialising $m_{0 \mid 0}$ and $C_{0 \mid 0}$ according to \Cref{equation-state-space-model} and set $G_{0 \mid 0} = I$.
Then, iterate from $(m_{0\mid k-1},C_{0\mid k-1},G_{0\mid k-1})$ to $(m_{0\mid k},C_{0\mid k},G_{0\mid k})$ using the recursion \citep{sarkka2010gaussian,sarkka2013bayesian:firstedition}
\begin{subequations}
\label{equation-fixed-point-smoother-recursion}
\begin{align}
G_{0 \mid k} &= G_{0 \mid k-1} G_{k-1 \mid k} \\
m_{0 \mid k} &= m_{0 \mid k-1} + G_{0 \mid k}(m_{k \mid k} - m_{k \mid k-1}) \\
C_{0 \mid k} &= C_{0 \mid k-1} + G_{0 \mid k}(C_{k \mid k} - C_{k \mid k-1}) (G_{0 \mid k})^\top.
\end{align}
\end{subequations}
Next to $(m_{0\mid k-1}, C_{0\mid k-1}, G_{0\mid k-1})$, these formulas only depend on the output of the prediction step as well as the smoothing gains $G_{k-1 \mid k}$ (smoothing gains: \Cref{appendix-section-rts-smoother}).
For $k=1$, the fixed-point gain intentionally equals the smoothing gain, which \Cref{equation-fixed-point-smoother-recursion} expresses by initialising $G_{0 \mid 0}=I$.
The recursion in \Cref{equation-fixed-point-smoother-recursion} can be implemented to run simultaneously with the forward filtering pass.
Unfortunately, even though this implementation avoids doubling the size of the state-space model and enjoys $\mathcal{O}(D^2)$ memory and $\mathcal{O}(K D^3)$ runtime, a Cholesky-based formulation has been unknown (until now).
Thus, \Cref{equation-fixed-point-smoother-recursion} cannot be applied to problems where numerical robustness is critical, like the probabilistic numerical simulation of differential equations \citep{kramer2024implementing}.
The main contribution of this paper is to generalise \Cref{equation-fixed-point-smoother-recursion} to enable Cholesky-based parametrisations:

\begin{contribution}
Compute the solution of the fixed-point smoothing problem $p(x_0 \mid y_{1:K})$ in $\mathcal{O}(K D^3)$ runtime, $\mathcal{O}(D^2)$ memory, using any parametrisation of Gaussian variables, and without state augmentation; \Cref{table-existing-approaches-to-fixed-point-smoothing}.
\end{contribution}
\begin{table}[t]
\caption{Existing approaches to the fixed-point smoothing problem. The ``fixed-point recursion'' represents \citet{meditch1969stochastic,meditch1967orthogonal} through  \Cref{equation-fixed-point-smoother-recursion}.
``$\mathcal{O}(1)$ memory'' stands in contrast to the $\mathcal{O}(K)$ memory of a fixed-interval smoother, and the ``low-dimensional state'' refers to avoiding state-augmentation, which affects both runtime and memory because it doubles the dimension of the state-space model.
}
\label{table-existing-approaches-to-fixed-point-smoothing}
\begin{center}
\begin{tabular}{ l c c c }
\toprule
Method & $\mathcal{O}(1)$ memory & Low-dimensional state & Cholesky-based
\\
\midrule
Via Rauch--Tung--Striebel smoothing& \xmark & \cmark & \cmark 
\\
Via filtering with augmented state & \cmark & \xmark & \cmark 
\\
Fixed-point recursion (\Cref{equation-fixed-point-smoother-recursion}) & \cmark & \cmark & \xmark 
\\
\midrule
This work (\Cref{algorithm-fixed-point-smoother,algorithm-covariance-based-implementation,algorithm-cholesky-based-implementation}) & \cmark & \cmark & \cmark
\\
\bottomrule
\end{tabular}
\end{center}
\end{table}

\section{The method: Numerically robust fixed-point smoothing}
\label{section-method}

Our approach to numerically robust fixed-point smoothing involves two steps:
First, we derive a recursion for the conditional distribution $p(x_0 \mid x_k, y_{1:k})$ instead of one for the joint distribution $p(x_0, x_k \mid y_{1:k})$, $k=1, ..., K$.
Second, we implement the recursion in Cholesky-based parametrisations without losing closed-form, constant-memory updates.
As a byproduct of this derivation, other parametrisations of Gaussian distributions (like the information form) become possible, too.
However, we focus on Cholesky-based implementations for their numerical robustness and leave other choices to future work.

\subsection{A new recursion for fixed-point smoothing}
A derivation of a new fixed-point smoothing recursion follows.
Similar to the state-augmented filtering perspective, the target distribution $p(x_0 \mid y_{1:K})$ can be written as the marginal of the augmented state
\begin{align}
\label{equation-final-marginalisation-fixedpoint}
p(x_0 \mid y_{1:K}) = \int p(x_0, x_K \mid y_{1:K}) \diff x_K.
\end{align}
Computing $p(x_0 \mid y_{1:K})$ becomes a byproduct of computing $p(x_0, x_K \mid y_{1:K})$, and the latter admits a sequential formulation.
This perspective was essential to implementing fixed-point smoothing with state-augmented filtering.
The augmented state factorises into the conditional
\begin{align}
p(x_0, x_K \mid y_{1:K})
= 
p(x_0 \mid x_K, y_{1:K}) p(x_K \mid y_{1:K}).
\end{align}
Since $p(x_K \mid y_{1:K})$ is already computed in the forward pass, it suffices to derive a recursion for the fixed-point conditional $p(x_0 \mid x_K, y_{1:K})$.
The backward factorisation of the smoothing distribution in \Cref{equation-smoothing-backward-factorisation} implies
\begin{align}
\label{equation-many-variable-integral}
p( x_0 \mid x_K, y_{1:K} ) 
=
\int 
\prod_{k=1}^K
	p( x_{k-1} \mid x_{k}, y_{1:k-1}) \diff x_{1:{K-1}}.
\end{align} 
Marginalising over all intermediate states $x_{1:K-1}$ like in \Cref{equation-many-variable-integral} turns a Rauch--Tung--Striebel smoother into a fixed-point smoother.
An $\mathcal{O}(K)$ runtime implementation now emerges from observing how the integral in \Cref{equation-many-variable-integral} rearranges to a nested {sequence} of single-variable integrals,
\begin{subequations}
\begin{align}
p(x_0 \mid x_K, y_{1:K})
&=
\int 
\prod_{k=1}^{K}
	p(x_{k-1} \mid x_{k}, y_{1:k-1}) \diff x_{1:{K-1}}
\\
&=
\int ... \left(\int\left[\int 
	p(x_0 \mid x_{1})
	p(x_1 \mid x_2, y_{1:1}) \diff x_1\right]
	p(x_2 \mid x_3, y_{1:2}) \diff x_2 \right) ... \diff x_{K-1}
\\
&=
\int ... 
\left(
\int
p(x_0 \mid x_2, y_1)
p(x_2 \mid x_3, y_{1:2})
\diff x_2
\right)
...
\diff x_{K-1}
\end{align}
\end{subequations}
This rearranging of the integrals is essential to the derivation because it implies a forward-in-time recursion:
\begin{algo}[Fixed-point smoother]
\label{algorithm-fixed-point-smoother}
To compute the solution to the fixed-point smoothing problem, assemble $p(x_K \mid y_{1:K})$ with a Kalman filter and evaluate $p(x_0 \mid x_K, y_{1:K})$ as follows. (To simplify the index-related notation in this algorithm, read $y_{1:-1} = y_{1:0} = \emptyset$.)
\begin{enumerate}
	\item 
	Initialise $p(x_0 \mid x_{0}, y_{1:-1})= \mathcal{N}(G_{0 \mid 0} x_0 + p_{0 \mid 0}, P_{0 \mid 0})$ with $G_{0 \mid 0} = I_D$, $p_{0 \mid 0} = 0$, and $P_{0 \mid 0} = 0$.
	\item For $k=1, ..., K$, iterate from $k-1$ to $k$,
	\begin{align}
	\label{equation-iteration-sqrt-fixed-point smoother}
	p(x_0 \mid x_{k}, y_{1:k-1}) 
	&= 
	\int 
	p(x_0 \mid x_{k-1}, y_{1:k-2}) 
	p(x_{k-1} \mid x_{k}, y_{1:k-1}) 
	\diff x_{k-1}.
	\end{align}
	The conditionals $p(x_{k-1} \mid x_{k}, y_{1:k-1})$ are from the Rauch--Tung--Striebel smoother (\Cref{equation-smoothing-backward-factorisation}). 
	\item Marginalise $p(x_0 \mid y_{1:K})$ from $p(x_{K} \mid y_{1:K})$ and $p(x_0 \mid x_K, y_{1:K-1})$ via \Cref{equation-final-marginalisation-fixedpoint}.
\end{enumerate}
\end{algo}
\Cref{equation-iteration-sqrt-fixed-point smoother} turns a Rauch--Tung--Striebel smoother into a fixed-point smoother:
The recursion requires access to $p(x_{k-1} \mid x_k, y_{1:k-1})$ computed by the forward-pass of the Rauch--Tung--Striebel smoother.
Therefore, \Cref{equation-iteration-sqrt-fixed-point smoother} runs concurrently to the forward filtering pass.

\subsection{Cholesky-based implementation}
The missing link for \Cref{algorithm-fixed-point-smoother} is the merging of two affinely related conditionals in \Cref{equation-iteration-sqrt-fixed-point smoother}.
Since both conditionals in \Cref{equation-iteration-sqrt-fixed-point smoother} are affine and Gaussian, their combination $p(x_0 \mid x_k, y_{1:k-1})$ is Gaussian and available in closed form for all $k=1, ..., K$.
Its parametrisation depends on the parametrisation of each input distribution.
For covariance- and Cholesky-based arithmetic, it looks as follows:
\begin{algo}[Covariance-based implementation of \Cref{equation-iteration-sqrt-fixed-point smoother}]
\label{algorithm-covariance-based-implementation}
Recall the initialisation of \Cref{algorithm-fixed-point-smoother} and the convention $y_{1:-1} =y_{1:0} = \emptyset$.
For any $k=1, ..., K$, if the fixed-point and smoothing conditionals (see \Cref{equation-smoothing-backward-factorisation} \& \Cref{appendix-section-rts-smoother} for the latter) are parametrised by
\begin{subequations}
\begin{align}
p(x_0 \mid x_{k-1}, y_{1:k-2}) 
	&= 
	\mathcal{N}(G_{0 \mid k-1} x_{k-1} + p_{0 \mid k-1}, P_{0 \mid k-1})
\\
p(x_{k-1} \mid x_k, y_{1:k-1}) 
	&= 
	\mathcal{N}(G_{k-1 \mid k} x_{k} + p_{k-1 \mid k}, P_{k-1 \mid k}),
\end{align}
\end{subequations}
for given $G_{i\mid j} \in \Rbb^{D \times D}$, $p_{i \mid j} \in \Rbb^{D}$, and $P_{i \mid j} \in \Rbb^{D \times D}$,
then, the next fixed-point conditional is
\begin{subequations}
\begin{align}
p(x_0 \mid x_k, y_{1:k-1})
=&\, \mathcal{N}( G_{0\mid k} x_k + p_{0\mid k}, P_{0 \mid k} ),
\\
G_{0 \mid k} 
\coloneqq&\, G_{0 \mid k-1} G_{k-1 \mid k}
, \\
p_{0 \mid k} \coloneqq&\,G_{0 \mid k-1} p_{k-1 \mid k} + p_{0 \mid k-1} 
, \\
P_{0 \mid k} \coloneqq&\,G_{0 \mid k-1} P_{k-1 \mid k} (G_{0 \mid k-1})^\top + P_{0 \mid k-1}.
\label{equation-cholesky-based-fixed-point-cov-in-cov}
\end{align}
\end{subequations}
Implementing this operation costs $\mathcal{O}(D^3)$ floating-point operations.
\end{algo}

\begin{algo}[Cholesky-based implementation of \Cref{equation-iteration-sqrt-fixed-point smoother}]
\label{algorithm-cholesky-based-implementation}
For any $k=1, ..., K$ (again \mbox{$y_{1:-1} = y_{1:0} = \emptyset$}), if the previous fixed-point conditional and smoothing transition (see \Cref{equation-smoothing-backward-factorisation} \& \Cref{appendix-section-rts-smoother}) are
\begin{subequations}
\begin{align}
p(x_0 \mid x_{k-1}, y_{1:k-2}) 
	&= 
	\mathcal{N}(G_{0 \mid k-1} x_{k-1} + p_{0 \mid k-1}, \Chol{P_{0 \mid k-1}} (\Chol{P_{0 \mid k-1}})^\top)
\\
p(x_{k-1} \mid x_k, y_{1:k-1}) 
	&= 
	\mathcal{N}(G_{k-1 \mid k} x_{k} + p_{k-1 \mid k}, \Chol{P_{k-1 \mid k}} (\Chol{P_{k-1 \mid k}})^\top ),
\end{align}
\end{subequations}
for given $G_{i \mid j} \in \Rbb^{D \times D}$, $p_{i \mid j} \in \Rbb^D$, and $\Chol{P_{i \mid j}} \in \Rbb^{D \times D}$,
then, the next fixed-point conditional is
\begin{subequations}
\begin{align}
p(x_0 \mid x_k, y_{1:k-1})
=&\, \mathcal{N}( G_{0\mid k} x_k + p_{0\mid k}, \Chol{P_{0 \mid k}} (\Chol{P_{0 \mid k}})^\top ),
\\
G_{0 \mid k} 
\coloneqq&\, G_{0 \mid k-1} G_{k-1 \mid k}
, \\
p_{0 \mid k} \coloneqq&\,G_{0 \mid k-1} p_{k-1 \mid k} + p_{0 \mid k-1} 
, \\
\Chol{P_{0 \mid k}} \coloneqq&\,\mathfrak{R}^\top,
\label{equation-cholesky-based-fixed-point-cov-in-cholesky}
\end{align}
\end{subequations}
where $\mathfrak{R}$ is the upper triangular matrix returned by the QR-decomposition
\begin{align}
\mathfrak{Q} \, \mathfrak{R} =
\begin{pmatrix}
(\Chol{P_{k-1 \mid k}})^\top (G_{0 \mid k-1})^\top
\\
(\Chol{P_{0 \mid k-1}})^\top
\end{pmatrix}.
\end{align}
Implementing this operation costs $\mathcal{O}(D^3)$ floating-point operations.
\end{algo}
\Cref{algorithm-covariance-based-implementation} follows from the rules for manipulating affinely related Gaussian distributions \citep[e.g.][Appendix A.1]{sarkka2013bayesian:firstedition}.
\Cref{algorithm-cholesky-based-implementation} is indeed the Cholesky-based version of \Cref{algorithm-covariance-based-implementation}:
The expression in \Cref{equation-cholesky-based-fixed-point-cov-in-cholesky} is the generalised Cholesky factor of the expression in \Cref{equation-cholesky-based-fixed-point-cov-in-cov} because
\begin{align}
\Chol{P_{k-1 \mid k}}
(\Chol{P_{k-1 \mid k}})^\top
=
\mathfrak{R}^\top \mathfrak{Q}^\top \mathfrak{Q} \mathfrak{R}
=
\begin{pmatrix}
G_{0 \mid k-1} \Chol{P_{k-1 \mid k}} 
&
\Chol{P_{0 \mid k-1}}
\end{pmatrix}
\begin{pmatrix}
(\Chol{P_{k-1 \mid k}})^\top (G_{0 \mid k-1})^\top
\\
(\Chol{P_{0 \mid k-1}})^\top
\end{pmatrix}
=
P_{k-1\mid k}
\end{align}
holds.
The combination of \Cref{algorithm-cholesky-based-implementation,algorithm-fixed-point-smoother}
is a novel, Cholesky-based implementation of
a fixed-point smoother.
Similarly, the combination of \Cref{algorithm-covariance-based-implementation,algorithm-fixed-point-smoother}
is a novel, covariance-based implementation of
a fixed-point smoother.
With a small modification to \Cref{algorithm-fixed-point-smoother}, combining \Cref{algorithm-covariance-based-implementation} with the covariance-based formulas in \Cref{algorithm-fixed-point-smoother} recovers \Cref{equation-fixed-point-smoother-recursion}:
\begin{proposition}
\label{proposition-reduction}
If the combination of \Cref{algorithm-covariance-based-implementation,algorithm-fixed-point-smoother} computes the marginals
\begin{align}
p(x_0 \mid y_{1:k})
=
\int p(x_0 \mid x_{k}, y_{1:k-1}) p(x_k \mid y_{1:k}) \diff x_k,
\end{align}
at every $k=1, ..., K$ instead of only at the final time-step, the recursion reduces to \Cref{equation-fixed-point-smoother-recursion}.
\end{proposition}
\begin{proof}
Induction; \Cref{appendix-section-proof-of-reduction-proposition}.
\end{proof}

\paragraph{Computational complexity}
Both \Cref{algorithm-covariance-based-implementation,algorithm-cholesky-based-implementation} cost $\mathcal{O}(D^3)$ floating-point operations per step.
The $k$th iteration in 
\Cref{algorithm-fixed-point-smoother}
needs access to the same quantities as the $k$th iteration of the forward pass of a Rauch--Tung--Striebel smoother, which can be implemented in $\mathcal{O}(D^2)$ memory.
Unlike the Rauch--Tung--Striebel smoother, \Cref{algorithm-fixed-point-smoother} does not store intermediate results, which is why it consumes $\mathcal{O}(D^2)$ memory in total (instead of $\mathcal{O}(K D^2)$).
Both
\Cref{algorithm-covariance-based-implementation} and \Cref{equation-fixed-point-smoother-recursion} require three matrix-matrix- and one matrix-vector multiplication per step, so they are equally efficient (\Cref{algorithm-covariance-based-implementation} uses fewer matrix additions. However, the addition of matrices is fast).
Therefore, it will be no loss of significance that we always implement covariance-based fixed-point smoothing via \Cref{algorithm-fixed-point-smoother} instead of \Cref{equation-fixed-point-smoother-recursion} in our experiments.

\clearpage
\section{Experiments}
\label{section-experiments}

The experiments serve two purposes.
To start with, they investigate whether the proposed Cholesky-based implementation (\Cref{algorithm-cholesky-based-implementation}) of the fixed-point smoother recursion (\Cref{algorithm-fixed-point-smoother}) holds its promises about memory, runtime, and numerical robustness.
According to the theory in \Cref{section-method}, we should observe:
\begin{itemize}
	\item \emph{Memory:} Slightly lower memory demands than a state-augmented Kalman filter; drastically lower memory demands than a Rauch--Tung--Striebel smoother.
	\item \emph{Wall-time:} Faster than a state-augmented, Cholesky-based Kalman filter; roughly as fast as a Rauch--Tung--Striebel smoother.
	\item \emph{Numerical robustness:} Combining \Cref{algorithm-fixed-point-smoother} with Cholesky-based parametrisations (\Cref{algorithm-cholesky-based-implementation}) is significantly more robust than combining it with covariance-based parametrisations (\Cref{algorithm-covariance-based-implementation}); comparable to a Cholesky-based implementation of a Kalman filter.
\end{itemize}
These three phenomena will be studied in two experiments, one for runtime/memory and one for numerical robustness (outline: \Cref{figure-outline-of-the-demonstrations}).
\begin{figure}
\begin{center}
\small
\begin{tikzpicture}

\node[minimum height=0.3cm, minimum width=3.6cm] (covb) { \sffamily \textbf{ Covariance-based} };
\node[ minimum height=0.3cm, minimum width=3.6cm, below=0.1cm of covb] (known1) { {Known} };
\node[ minimum height=0.3cm, minimum width=3.6cm, below=0.1cm of known1] (known2) { {Known} };
\node[ minimum height=0.3cm, minimum width=3.6cm, below=0.1cm of known2] (known3) { {Known, but see Prop. \ref{proposition-reduction}} };

\node[ minimum height=0.3cm, minimum width=3.6cm, right=0.2cm of covb] (cholb) { \sffamily \textbf{ Cholesky-based} };
\node[ minimum height=0.3cm, minimum width=3.6cm, below=0.1cm of cholb] (known4) { {Known} };
\node[ minimum height=0.3cm, minimum width=3.6cm, below=0.1cm of known4] (known5) { {Known} };
\node[minimum height=0.3cm, minimum width=3.6cm, below=0.1cm of known5] (ours) { { {Our contribution} } };

\node[ minimum height=0.3cm, minimum width=3.6cm, left=0.1cm of known1] (filter) { \sffamily \textbf{ Via filter} };
\node[ minimum height=0.3cm, minimum width=3.6cm, left=0.1cm of known2] (fixedinterval) { \sffamily \textbf{ Via Rauch--Tung--Striebel } };
\node[ minimum height=0.3cm, minimum width=3.6cm, left=0.1cm of known3] (fixed-point) { \sffamily \textbf{ Fixed-point recursion} };

\node[minimum height=0.3cm, minimum width=4.0cm, right=0.4cm of ours] (demoI) { \parbox{4.0cm}{  Experiment II: Robustness} };
\node[minimum height=0.3cm, minimum width=3.6cm, below=0.4cm of ours] (demoII) {  { Experiment I: Efficiency} };

\node[gray, fill, opacity=0.1, fit=(known4) (demoII)] {};
\node[gray, fill, opacity=0.1, fit=(known3) (demoI)] {};

\end{tikzpicture}
\end{center}
\caption{Outline of the memory-, runtime-, and robustness-related demonstrations.}
\label{figure-outline-of-the-demonstrations}
\end{figure}
The problem set for these two experiments includes a toy problem for the former and an application in probabilistic numerics for the latter.
Afterwards, a case study in tracking shows how to use the fixed-point smoother for estimating the initial parameter in a state-space model.

\paragraph{Hardware and code}
All experiments run on the CPU of a consumer-grade laptop and finish within a few minutes. 
Our JAX implementation \citep{jax2018github} of Kalman filters, 
Rauch--Tung--Striebel smoothers, and fixed-point smoothers is at
\begin{center}
\code{}
\end{center}
We implement the existing fixed-point smoother recursions in \Cref{equation-fixed-point-smoother-recursion} by combining \Cref{algorithm-fixed-point-smoother} with covariance-based parametrisations; recall \Cref{proposition-reduction}.

\subsection{Experiment I: How efficient is the fixed-point recursion?}

\paragraph{Motivation}
\Cref{section-existing-approaches-to-fixed-point-smoothing,section-method} mention three approaches to fixed-point smoothing: a detour via Rauch--Tung--Striebel smoothing, state-augmented Kalman filtering, and our recursion in \Cref{algorithm-fixed-point-smoother}.
In theory, \Cref{algorithm-fixed-point-smoother} should be the most efficient: it does not inflate the state-space model like a state-augmented Kalman filter does and requires $\mathcal{O}(1)$ instead of $\mathcal{O}(K)$ memory, unlike the Rauch--Tung--Striebel smoother.
This first experiment demonstrates that these effects are visible in practice.

\paragraph{Problem setup}
In this first example, we only measure the execution time and memory requirement of three approaches to fixed-point smoothing.
Both memory and runtime depend only on the size of a state-space model and not on the difficulty of the estimation task.
Therefore, we consider a state-space model where all system matrices and vectors are populated with random values.

\paragraph{Implementation}
We choose $K=1,000$, vary $d$, set the size of the hidden state to $D=2d$, and use Cholesky-based arithmetic for all estimators.
We take \Cref{equation-state-space-model} and introduce a nonzero bias in all noises in \Cref{equation-state-space-model-noise}.
Then, we randomly populate all system matrices in the state-space model with independent samples from $\mathcal{N}(0, \sfrac{1}{K^2})$.
Afterwards, we sample $y_{1:K}$ from this state-space model to generate toy data.
A $K$-dependent covariance controls that the samples $y_{1:K}$ remain finite in 32-bit floating-point arithmetic. However, the precise values of the model and data do not matter for this experiment -- only their size does.
Finally, we vary $d$ and measure the runtime and memory requirements of each of the three methods.
To measure runtime, we use wall time in seconds.
We display the best of three runs instead of the average because all codes are deterministic (no data-dependence, no randomness), thus the main driver for runtime differences is background machine noise.
Selecting the best of three runs minimises this noise as much as possible.
To measure memory consumption, we count the number of floating-point values the estimators carry from step to step, multiply this by 32 (because we use 32-bit arithmetic), and translate bits into bytes.

\paragraph{Evaluation}
The runtime results are in \Cref{table-results-runtime} and the memory results in  \Cref{table-results-memory}.
\begin{table}[t]
\caption{%
Runtime in seconds (wall-time, best of three runs).
Lower is better. The column-wise lowest are bold \& shaded.
Randomly populated model.
$K=1,000$ steps.
All methods use Cholesky-based parametrisations.
}
\label{table-results-runtime}
\begin{center}
\small
\begin{tabular}{l c c c c c c}
\toprule
 & $d=2$ & $d=5$ & $d=10$ & $d=20$ & $d=50$ & $d=100$ \\
\midrule
Via Rauch--Tung--Striebel &  ${5.8 \times 10^{-3}}$ & $\winner{1.8 \times 10^{-2}} $ & $6.5 \times 10^{-2}$ & $4.5 \times 10^{-1}$ & $\winner{2.9 \times 10^{0}}$ & $\winner{1.2 \times 10^{1}}$ \\
Via filter & $\winner{ 5.0 \times 10^{-3}} $ & $2.1 \times 10^{-2}$ & $8.0 \times 10^{-2}$ & $7.1 \times 10^{-1}$ & $4.6 \times 10^{0}$ & $1.6 \times 10^{1}$ \\
\midrule
\Cref{algorithm-fixed-point-smoother} & $6.4 \times 10^{-3}$ & $\winner{1.8 \times 10^{-2}}$ & $\winner{6.4 \times 10^{-2}}$ & $\winner{4.4 \times 10^{-1}}$ & $\winner{2.9 \times 10^{0}}$ & $\winner{1.2 \times 10^{1}}$ \\
\bottomrule
\end{tabular}

\end{center}
\end{table}
\begin{table}[t]
\caption[Caption with footnote]{%
Memory in bytes (we use 32-bit arithmetic).
Lower is better. The column-wise lowest entries are bold and shaded.
Randomly populated model. $K=1,000$ steps.
All methods use Cholesky-based parametrisations.
}
\label{table-results-memory}
\begin{center}
\small
\begin{tabular}{ l c c c c c c }
\toprule
 & $d=2$ & $d=5$ & $d=10$ & $d=20$ & $d=50$ & $d=100$ \\
\midrule
Via Rauch--Tung--Striebel & $2.2 \times 10^{5}$ & $1.2 \times 10^{6}$ & $4.9 \times 10^{6}$ & $1.9 \times 10^{7}$ & $1.2 \times 10^{8}$ & $4.8 \times 10^{8}$ \\
Via filter & $2.8 \times 10^{2}$ & $1.6 \times 10^{3}$ & $6.5 \times 10^{3}$ & $2.5 \times 10^{4}$ & $1.6 \times 10^{5}$ & $6.4 \times 10^{5}$ \\
\midrule
\Cref{algorithm-fixed-point-smoother} & $\winner{2.2 \times 10^{2}}$ & $\winner{1.2 \times 10^{3}}$ & $\winner{4.9 \times 10^{3}}$ & $\winner{1.9 \times 10^{4}}$ & $\winner{1.2 \times 10^{5}}$ & $\winner{4.8 \times 10^{5}}$ \\
\bottomrule
\end{tabular}
\end{center}
\end{table}
The runtime data shows that except for $d=2$, the Rauch--Tung--Striebel smoother and \Cref{algorithm-fixed-point-smoother} are faster than the state-augmented filter.
\Cref{algorithm-fixed-point-smoother} is marginally faster than the Rauch--Tung--Striebel smoother code, even though it computes strictly more at every iteration.
However, the Rauch--Tung--Striebel smoother executes two loops, one forwards and one backwards, whereas \Cref{algorithm-fixed-point-smoother} only executes a forward loop.
This discrepancy might lead to the performance improvement from Rauch--Tung--Striebel smoothing to \Cref{algorithm-fixed-point-smoother}.
The memory data shows that the Rauch--Tung--Striebel smoothing solution and the \Cref{algorithm-fixed-point-smoother} have identical memory requirements per step.
Both consume less storage than the state-augmented filter (per step).
As expected, the Rauch--Tung--Striebel smoother requires exactly $K$-times the memory of \Cref{algorithm-fixed-point-smoother}, which would be infeasible for long time series over high-dimensional states.
In general, this experiment underlines the claimed efficiency of our new fixed-point smoother recursion in a Cholesky-based parametrisation.
The following experiment will answer the question of whether Cholesky-based arithmetic is necessary.

\subsection{Experiment II: How much more robust is the Cholesky-based code?}
\label{section-experiment-robustness-bvp}

\paragraph{Motivation}
Cholesky-based implementations replace matrix addition and matrix multiplication with QR decompositions.
They are more expensive than covariance-based implementations because the decomposition is more expensive than matrix multiplication.
However, there are numerous situations where the gains in numerical robustness are worth the increase in runtime.
One example is the probabilistic numerical simulation of differential equations, which use integrated Wiener processes \citep{schober2014probabilistic,schober2019probabilistic} with high order and small time-steps.
For these applications, Cholesky-based implementations are the only feasible approach \citep{kramer2024stable}.
This second experiment shows that \Cref{algorithm-fixed-point-smoother} unlocks fixed-point smoothing for probabilistic numerical simulation through \Cref{algorithm-cholesky-based-implementation}.

\paragraph{Problem setup}
We solve a boundary value problem based on an ordinary differential equation.
\begin{wrapfigure}[12]{r}{0.32\linewidth}
	\vskip-1.25em
	\includegraphics[width=0.97\linewidth]{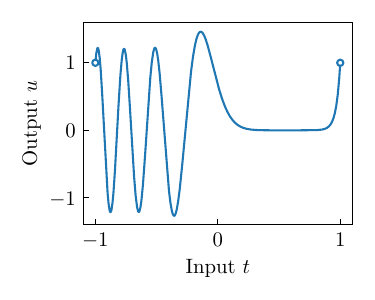}
	\vskip-1.5em
	\caption{15th Boundary value problem \citep{mazzia2015fortran}.}
	\label{figure-bvp-solution}
\end{wrapfigure}
More specifically, we solve the 15th in the collection of test problems by \citet{mazzia2015fortran} (\Cref{figure-bvp-solution}),
\begin{align}
10^{-3} \cdot \oderiv{2} u(t) = t u(t), \quad u(-1) = u(1) = 1.  
\end{align}
We follow the procedure for solving boundary value problems with a probabilistic numerical method by \citet{kramer2021linear} for the most part.
However, we skip the iterated Kalman smoother because the differential equation is linear and skip mesh refinement and expectation maximisation to keep this demonstration simple.
The focus lies on numerical robustness.
We choose a twice-integrated Wiener process prior and discretise it on $K$ equispaced points in $[-1, 1]$, which yields the latent transitions (let $\Delta t \coloneqq \sfrac{1}{K}$)
\begin{align}
A_k \coloneqq
\begin{pmatrix}
1 & \Delta t & \frac{(\Delta t)^2}{2} \\
0 & 1 & \Delta t \\
0 & 0 & 1
\end{pmatrix}
,\quad
B_k \coloneqq
\begin{pmatrix}
\frac{(\Delta t)^5}{20} & \frac{(\Delta t)^4}{8} & \frac{(\Delta t)^3}{3} \\
\frac{(\Delta t)^4}{8} & \frac{(\Delta t)^3}{3} & \frac{(\Delta t)^2}{2} \\
\frac{(\Delta t)^3}{3} & \frac{(\Delta t)^2}{2} & \Delta t \\
\end{pmatrix},
\quad
k=1, ..., K.
\end{align}
The dynamics are constant because the grid points are evenly spaced.
The means of the process noise are zero, like in \Cref{equation-state-space-model-noise}.
The hidden state 
\begin{align}
x_k \coloneqq \left(u(t_k), \oderiv{}u(t_k), \oderiv{2}u(t_k) \right), \quad k=1, ..., K
\end{align}
tracks the differential equation solution $u$ and its derivatives, including $\oderiv{2}u$, which means that the residual $10^{-3} \oderiv{2} u - t u(t)$ is a linear function of $x_k$.
We introduce the model for the constraints
($\beta_k$ shall be the nonzero mean of the observation noise $r_k$ in \Cref{equation-state-space-model-noise,equation-state-space-model})
\begin{subequations}
\begin{align}
H_k &= 
\begin{pmatrix}
-t_k & 0 & 10^{-3}
\end{pmatrix}
, &
\beta_k &= 0
, &
R_k &= 0
, &
k=1, ..., K-1,
\\
H_K &= \begin{pmatrix} 1 & 0 & 0 \end{pmatrix},
&
\beta_K &= -1,
&
R_K &= 0.
\end{align}
\end{subequations}
Note how the constraint at the final time-point encodes the right-hand side boundary condition, not the differential equation.
We choose the initial mean $m_{0 \mid 0} = (1, 0, 0)$ and initial covariance $C_{0 \mid 0} = \text{diag}(0, 1, 1)$ to represent the left-hand side boundary condition.
Estimating $x_{0:K}$ from $y_{1:K} \coloneqq 0$ solves the boundary value problem (\citep{kramer2021linear}; \Cref{figure-bvp-solution} displays the mean of the fixed-interval smoothing solution).
\citet{kramer2021linear} explain how estimating the initial condition $p(x_0 \mid y_{1:K})$ is important for parameter estimation problems.
In other words, fixed-point smoothers are relevant for boundary value problem simulation.
We are the first to use them for this task.

\paragraph{Implementation}
For this demonstration, we consider the Cholesky-based implementation of the Kalman filter as the gold standard for numerical robustness because it has been used successfully in numerous similar problems \citep{kramer2024stable,bosch2021calibrated,kramer2021linear,kramer2024implementing}.
The Cholesky-based, state-augmented Kalman filter provides a reference solution of the fixed-point equations.
We run Cholesky-based (\Cref{algorithm-cholesky-based-implementation}) and covariance-based code (\Cref{algorithm-covariance-based-implementation}) for the fixed-point smoother recursions (\Cref{algorithm-fixed-point-smoother}) and measure how much the estimated initial means deviate from the Kalman-filter reference in terms of the absolute root-mean-square error.
In infinite precision, the results would be identical.
In finite precision, all deviations should be due to a loss of stability.
We vary $K$ to investigate how the step-size $\Delta t$ affects the results, with the intuition that smaller steps lead to worse conditioning in $B_k$ and that this effect makes the estimation task more difficult \citep{kramer2024stable}.

\paragraph{Evaluation}
The results are in \Cref{table-mean-deviation-numerical-robustness}.
\begin{table}
\caption{Deviation from the Cholesky-based, state-augmented Kalman filter on the boundary value problem. Lower is better and close to machine-precision desirable.
The column-wise lowest values are bold and coloured. ``$K$'': number of grid points. Double precision (64-bit arithmetic).
}
\label{table-mean-deviation-numerical-robustness}
\begin{center}
\small
\begin{tabular}{l c c c c c c c}
\toprule
*-based:  & $K=10$ & $K=20$ & $K=50$ & $K=100$ & $K=200$ & $K=500$ & $K=1,000$ \\
\midrule
Covariance & $2.9 \times 10^{-3}$ & $3.4 \times 10^{-1}$ & $1.1 \times 10^{0}$ & $1.0 \times 10^{1}$ & $2.1 \times 10^{1}$ & NaN & NaN \\
Cholesky & $\winner{2.0 \times 10^{-10}}$ & $\winner{5.0 \times 10^{-8}}$ & $\winner{4.2 \times 10^{-7}}$ & $\winner{7.9 \times 10^{-8}}$ & $\winner{1.3 \times 10^{-7}}$ & $\winner{6.1 \times 10^{-8}}$ & $\winner{3.4 \times 10^{-8}}$ \\
\bottomrule
\end{tabular}
\end{center}
\end{table}
They indicate how the Cholesky-based code is significantly more robust than the covariance-based code.
The covariance-based implementation delivers only three meaningful digits for $K=10$ points, diverges for more grid points, and breaks for $K\geq 500$.
In contrast, the Cholesky-based code delivers meaningful approximations across all grids.
This consistency underlines how much more robust our Cholesky-based code is and how it unlocks fixed-point smoothing for probabilistic numerics.

\subsection{Case study: Estimating parameters of a state-space model}

\paragraph{Movitation}
The previous two experiments have emphasized the practicality of our proposed method. We conclude by applying fixed-point smoothing to parameter estimation in a tracking model.
This study aims to emulate applications of fixed-point smoothing in navigation tasks \citep{meditch1969stochastic} in a setup that is easy to reproduce. However, the demonstration also links to the previous boundary value problem experiment through expectation maximisation \citep{kramer2021linear}.

\paragraph{Problem setup}
The task is to estimate the mean parameter of an unknown initial condition in a car tracking example.
Define the Wiener velocity model on $K=10$ equispaced points ($\Delta t = \sfrac{1}{10})$, 
\begin{align}
A_k 
\coloneqq
\begin{pmatrix}
I_2 & \Delta t \cdot I_2 \\
0 & I_2
\end{pmatrix}
\quad
B_k  \coloneqq
\begin{pmatrix}
\frac{(\Delta t)^3}{3}\cdot I_2 & \frac{(\Delta t)^2}{2}\cdot I_2 \\
\frac{(\Delta t)^2}{2} \cdot I_2& \Delta t\cdot I_2
\end{pmatrix}
,\quad
H_k  = 
\begin{pmatrix}
I_2 & 0
\end{pmatrix}
,\quad
R_k  = 0.1^2 \cdot I_2.
\end{align}
All biases are zero, $b_k = 0$, $r_k = 0$.
Textbooks on Bayesian filtering and smoothing \citep{sarkka2013bayesian:firstedition,sarkka2023bayesian} use this Wiener velocity model to estimate the trajectory of a car. 
We populate $m$ and $\Chol{C}$ with samples from a standard normal distribution and sample an initial condition 
\begin{align}
x_0 = 
\begin{pmatrix} \theta_1 & \theta_2 & \dot \theta_1 & \dot \theta_2 \end{pmatrix} 
\sim \mathcal{N}(m_{0 \mid 0}, \Chol{C_{0 \mid 0}} (\Chol{C_{0 \mid 0}})^\top).
\end{align}
We sample artificial observations $y_{1:K}$.
From here on, the initial mean $m_{0 \mid 0}$ will be treated as unknown.

\paragraph{Implementation}
We combine fixed-point smoothing with expectation maximisation \citep{dempster1977maximum} to calibrate $m_{0 \mid 0}$ using the data $y_{1:K}$.
The expectation maximisation update for the initial mean in a linear Gaussian state-space model is $m_\text{new} \coloneqq m_{0 \mid K}
$ \citep{sarkka2023bayesian}, and this update repeats until convergence.
Here, $m_{0 \mid K}$ is the mean of $p(x_0 \mid y_{1:K})$.
We implement the fixed-point smoother recursion in Cholesky-based arithmetic and run expectation maximisation for three iterations. We initialise the mean guess by sampling all entries independently from a centred normal distribution with a variance of 100.
We track the data's marginal likelihood (``evidence''), computing it online during the forward filtering pass.

\paragraph{Evaluation}
The results are in \Cref{figure-results-for-parameter-estimation}.
\begin{figure}[t]
\begin{center}
\includegraphics[width=\linewidth]{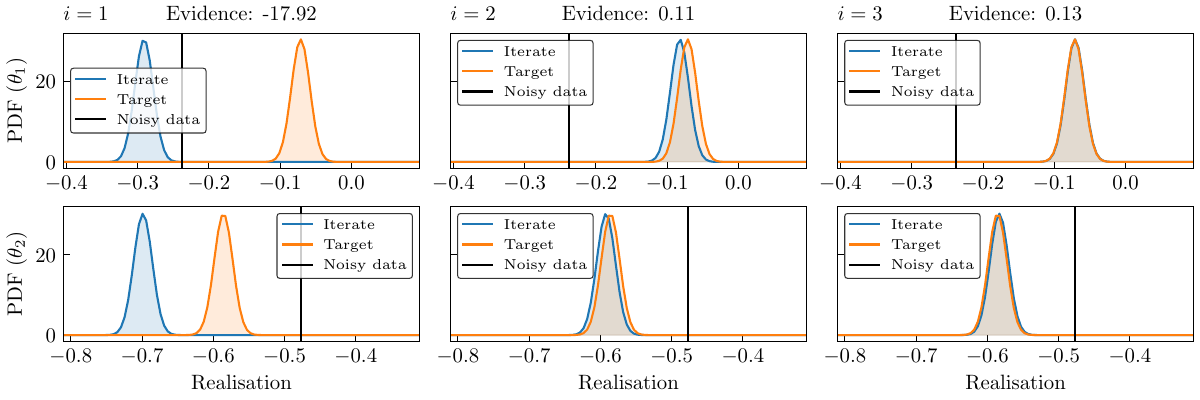}
\end{center}
\caption{%
Initial distributions $p(\theta_1, \theta_2 \mid y_{1:K})$ of the car tracking model after running the fixed-point smoother (recall $x\coloneqq(\theta_1,  \theta_2, \dot \theta_1, \dot \theta_2)$).
Left to right:
After three iterations, the combination of expectation maximisation with the fixed-point smoother finds the correct initial mean $\theta = (\theta_1, \theta_2)$ of the state-space model. 
Top: First coordinate $p(\theta_1 \mid y_{1:K})$. Bottom: second coordinate $p(\theta_2 \mid y_{1:K})$.
``PDF'': ``Probability density function''.}
\label{figure-results-for-parameter-estimation}
\end{figure}
They show how the combination of expectation maximisation with fixed-point smoothing recovers the initial mean already after three iterations.
The evidence increases at every iteration, but that is normal for expectation maximisation \citep{wu1983convergence}.
In conclusion, fixed-point smoothing is a viable parameter estimation technique in a state-space model.

\section{Discussion}
\label{section-discussion}

\paragraph{Limitations and future work}
Our new approach to fixed-point smoothing strictly generalises existing techniques, but inherits some of their limitations:
Even though the new recursion in \Cref{algorithm-fixed-point-smoother} is independent of the type of state-space model, implementing the fixed-point smoother in closed form assumes a linear Gaussian setup.
Future work should explore robust fixed-point smoothing in nonlinear state-space models, for example, through posterior linearisation \citep{garcia2016iterated}.
Like all Gaussian smoothing algorithms, our methods have cubic complexity in the state-space dimension because of the matrix-matrix arithmetic or QR decompositions, respectively.
Finally, other parametrisations of Gaussian variables could be used with \Cref{algorithm-fixed-point-smoother} instead of covariance or Cholesky-based parametrisations (\Cref{algorithm-covariance-based-implementation,algorithm-cholesky-based-implementation}); for example, the information form or ensembles \citep{murphy2022probabilistic,houtekamer2005ensemble}. Combining these alternatives with numerically robust fixed-point smoothing could be interesting avenues for future work.

\paragraph{Conclusion}
This paper presented a new recursion for fixed-point smoothing in Gaussian state-space models: \Cref{algorithm-fixed-point-smoother}.
It has lower memory consumption than existing approaches because it avoids state-augmentation and foregoes storing all intermediate results of a Rauch--Tung--Striebel smoother.
It also allows arbitrary parametrisations of Gaussian distributions, and we use this perspective to derive a Cholesky-based form of fixed-point smoothers.
As a result, our method matches the speed of the fastest and the robustness of the most robust methods, as has been demonstrated in three simulations of varying difficulty.
Through these successes, our contribution hopefully revives fixed-point smoothing as part of the literature on Gaussian filters and smoothers.
We anticipate notable performance improvements for algorithms at the interface of smoothing and dynamical systems due to the newfound ability to reliably use non-standard smoothing algorithms.

\acks{}

\bibliography{main.bib}

\newcommand{\arxiv}[1]{Preprint on ArXiv:#1}
\begin{thebibliography}{43}
\providecommand{\natexlab}[1]{#1}
\providecommand{\url}[1]{\texttt{#1}}
\expandafter\ifx\csname urlstyle\endcsname\relax
  \providecommand{\doi}[1]{doi: #1}\else
  \providecommand{\doi}{doi: \begingroup \urlstyle{rm}\Url}\fi

\bibitem[Grewal and Andrews(2014)]{grewal2014kalman}
Mohinder~S Grewal and Angus~P Andrews.
\newblock \emph{{Kalman} filtering: Theory and Practice with MATLAB}.
\newblock John Wiley \& Sons, 2014.

\bibitem[S{\"a}rkk{\"a} and Svensson(2023)]{sarkka2023bayesian}
Simo S{\"a}rkk{\"a} and Lennart Svensson.
\newblock \emph{Bayesian Filtering and Smoothing}.
\newblock Cambridge University Press, 2023.

\bibitem[S{\"a}rkk{\"a} and Solin(2019)]{sarkka2019applied}
Simo S{\"a}rkk{\"a} and Arno Solin.
\newblock \emph{Applied Stochastic Differential Equations}, volume~10.
\newblock Cambridge University Press, 2019.

\bibitem[Murphy(2023)]{murphy2023probabilistic}
Kevin~P Murphy.
\newblock \emph{Probabilistic Machine Learning: Advanced Topics}.
\newblock MIT Press, 2023.

\bibitem[Hennig et~al.(2022)Hennig, Osborne, and
  Kersting]{hennig2022probabilistic}
Philipp Hennig, Michael~A Osborne, and Hans~P Kersting.
\newblock \emph{Probabilistic Numerics: Computation as Machine Learning}.
\newblock Cambridge University Press, 2022.

\bibitem[Singhal and Wu(1988)]{singhal1988training}
Sharad Singhal and Lance Wu.
\newblock Training multilayer perceptrons with the extended kalman algorithm.
\newblock \emph{Advances in neural information processing systems}, 1, 1988.

\bibitem[Gu and Dao(2023)]{gu2023mamba}
Albert Gu and Tri Dao.
\newblock Mamba: Linear-time sequence modeling with selective state spaces.
\newblock \emph{arXiv preprint arXiv:2312.00752}, 2023.

\bibitem[Chang et~al.(2023)Chang, Dur{\'a}n-Mart{\'\i}n, Shestopaloff, Jones,
  and Murphy]{chang2023lowrank}
Peter~G Chang, Gerardo Dur{\'a}n-Mart{\'\i}n, Alex Shestopaloff, Matt Jones,
  and Kevin~Patrick Murphy.
\newblock Low-rank extended kalman filtering for online learning of neural
  networks from streaming data.
\newblock In \emph{Conference on Lifelong Learning Agents}, pages 1025--1071.
  PMLR, 2023.

\bibitem[Sliwa et~al.(2024)Sliwa, Schneider, Bosch, Kristiadi, and
  Hennig]{sliwa2024efficient}
Joanna Sliwa, Frank Schneider, Nathanael Bosch, Agustinus Kristiadi, and
  Philipp Hennig.
\newblock Efficient weight-space {Laplace--Gaussian} filtering and smoothing
  for sequential deep learning.
\newblock \emph{arXiv preprint arXiv:2410.06800}, 2024.

\bibitem[Garc{\'\i}a-Fern{\'a}ndez et~al.(2016)Garc{\'\i}a-Fern{\'a}ndez,
  Svensson, and S{\"a}rkk{\"a}]{garcia2016iterated}
{\'A}ngel~F Garc{\'\i}a-Fern{\'a}ndez, Lennart Svensson, and Simo
  S{\"a}rkk{\"a}.
\newblock Iterated posterior linearization smoother.
\newblock \emph{IEEE Transactions of Automatic Control}, 62\penalty0
  (4):\penalty0 2056--2063, 2016.

\bibitem[S{\"a}rkk{\"a} and
  Garc{\'\i}a-Fern{\'a}ndez(2020)]{sarkka2020temporal}
Simo S{\"a}rkk{\"a} and {\'A}ngel~F Garc{\'\i}a-Fern{\'a}ndez.
\newblock Temporal parallelization of {Bayesian} smoothers.
\newblock \emph{IEEE Transactions of Automatic Control}, 66\penalty0
  (1):\penalty0 299--306, 2020.

\bibitem[Yaghoobi et~al.(2022)Yaghoobi, Corenflos, Hassan, and
  S{\"a}rkk{\"a}]{yaghoobi2022parallel}
Fatemeh Yaghoobi, Adrien Corenflos, Sakira Hassan, and Simo S{\"a}rkk{\"a}.
\newblock Parallel square-root statistical linear regression for inference in
  nonlinear state space models.
\newblock \emph{\arxiv{2207.00426}}, 2022.

\bibitem[Meditch(1969)]{meditch1969stochastic}
James~S Meditch.
\newblock \emph{Stochastic Optimal Linear Estimation and Control}.
\newblock McGraw-Hill, 1969.

\bibitem[Schober et~al.(2019)Schober, S{\"a}rkk{\"a}, and
  Hennig]{schober2019probabilistic}
Michael Schober, Simo S{\"a}rkk{\"a}, and Philipp Hennig.
\newblock A probabilistic model for the numerical solution of initial value
  problems.
\newblock \emph{Statistics and Computing}, 29\penalty0 (1):\penalty0 99--122,
  2019.

\bibitem[Kr{\"a}mer and Hennig(2024)]{kramer2024stable}
Nicholas Kr{\"a}mer and Philipp Hennig.
\newblock Stable implementation of probabilistic {ODE} solvers.
\newblock \emph{Journal of Machine Learning Research}, 25\penalty0
  (111):\penalty0 1--29, 2024.

\bibitem[Rackauckas et~al.(2020)Rackauckas, Ma, Martensen, Warner, Zubov,
  Supekar, Skinner, Ramadhan, and Edelman]{rackauckas2020universal}
Christopher Rackauckas, Yingbo Ma, Julius Martensen, Collin Warner, Kirill
  Zubov, Rohit Supekar, Dominic Skinner, Ali Ramadhan, and Alan Edelman.
\newblock Universal differential equations for scientific machine learning.
\newblock \emph{arXiv preprint arXiv:2001.04385}, 2020.

\bibitem[Biswas and Mahalanabis(1972)]{biswas1972approach}
KK~Biswas and AK~Mahalanabis.
\newblock An approach to fixed-point smoothing problems.
\newblock \emph{IEEE Transactions on Aerospace and Electronic Systems},
  \penalty0 (5):\penalty0 676--682, 1972.

\bibitem[Smith and Roberts(1982)]{smith1982robustness}
MWA Smith and AP~Roberts.
\newblock The robustness ot the fixed point smoothing algorithm.
\newblock \emph{Journal of Information and Optimization Sciences}, 3\penalty0
  (2):\penalty0 87--104, 1982.

\bibitem[Meditch(1967{\natexlab{a}})]{meditch1967optimal}
JS~Meditch.
\newblock Optimal fixed-point continuous linear smoothing.
\newblock In \emph{Joint Automatic Control Conference}, number~5, pages
  249--257, 1967{\natexlab{a}}.

\bibitem[Särkkä and Hartikainen(2010)]{sarkka2010gaussian}
Simo Särkkä and Jouni Hartikainen.
\newblock On {Gaussian} optimal smoothing of non-linear state space models.
\newblock \emph{IEEE Transactions of Automatic Control}, 55\penalty0
  (8):\penalty0 1938--1941, 2010.

\bibitem[Rauch(1963)]{rauch1963solutions}
H~Rauch.
\newblock Solutions to the linear smoothing problem.
\newblock \emph{IEEE Transactions of Automatic Control}, 8\penalty0
  (4):\penalty0 371--372, 1963.

\bibitem[Meditch(1967{\natexlab{b}})]{meditch1967orthogonal}
James~S Meditch.
\newblock Orthogonal projection and discrete optimal linear smoothing.
\newblock \emph{SIAM Journal on Control}, 5\penalty0 (1):\penalty0 74--89,
  1967{\natexlab{b}}.

\bibitem[Meditch(1976)]{meditch1976initial}
James~S Meditch.
\newblock Initial- and lagging-state observers.
\newblock \emph{Information Sciences}, 11\penalty0 (1):\penalty0 55--67, 1976.

\bibitem[Meditch and Hostetter(1973)]{meditch1973estimation}
JS~Meditch and GH~Hostetter.
\newblock Estimation of initial conditions in linear systems.
\newblock In \emph{Proceedings of the 11th Annual Atierton Conference on
  Circuit and System Theory}, pages 226--235, 1973.

\bibitem[Nishimura(1969)]{nishimura1969new}
T~Nishimura.
\newblock A new approach to estimation of initial conditions and smoothing
  problems.
\newblock \emph{IEEE Transactions on Aerospace and Electronic Systems},
  \penalty0 (5):\penalty0 828--836, 1969.

\bibitem[Kr{\"a}mer and Hennig(2021)]{kramer2021linear}
Nicholas Kr{\"a}mer and Philipp Hennig.
\newblock Linear-time probabilistic solution of boundary value problems.
\newblock \emph{Advances in Neural Information Processing Systems},
  34:\penalty0 11160--11171, 2021.

\bibitem[Kr{\"a}mer(2024)]{kramer2024implementing}
Peter~Nicholas Kr{\"a}mer.
\newblock \emph{Implementing probabilistic numerical solvers for differential
  equations}.
\newblock PhD thesis, Universit{\"a}t T{\"u}bingen, 2024.

\bibitem[Bennett(1965)]{bennett1965triangular}
John~M Bennett.
\newblock Triangular factors of modified matrices.
\newblock \emph{Numerische Mathematik}, 7:\penalty0 217--221, 1965.

\bibitem[Murphy(2022)]{murphy2022probabilistic}
Kevin~P Murphy.
\newblock \emph{Probabilistic Machine Learning: an Introduction}.
\newblock MIT press, 2022.

\bibitem[Kalman(1960)]{kalman1960new}
Rudolph~E Kalman.
\newblock A new approach to linear filtering and prediction problems.
\newblock \emph{Journal of Basic Engineering}, 82\penalty0 (1):\penalty0
  35--45, 1960.

\bibitem[Andrews(1968)]{andrews1968square}
Angus Andrews.
\newblock A square root formulation of the kalman covariance equations.
\newblock \emph{AIAA Journal}, 6\penalty0 (6):\penalty0 1165--1166, 1968.

\bibitem[Rauch et~al.(1965)Rauch, Tung, and Striebel]{rauch1965maximum}
Herbert~E Rauch, F~Tung, and Charlotte~T Striebel.
\newblock Maximum likelihood estimates of linear dynamic systems.
\newblock \emph{AIAA journal}, 3\penalty0 (8):\penalty0 1445--1450, 1965.

\bibitem[Park and Kailath(1995)]{park1995square}
Poogyeon Park and Thomas Kailath.
\newblock Square-root {RTS} smoothing algorithms.
\newblock \emph{International Journal of Control}, 62\penalty0 (5):\penalty0
  1049--1060, 1995.

\bibitem[Särkkä(2013)]{sarkka2013bayesian:firstedition}
Simo Särkkä.
\newblock \emph{Bayesian Filtering and Smoothing}.
\newblock Cambridge University Press, 2013.
\newblock (Note: This is the first edition. The second edition
  \citep{sarkka2023bayesian} does not discuss fixed-point smoothing.).

\bibitem[Bradbury et~al.(2018)Bradbury, Frostig, Hawkins, Johnson, Leary,
  Maclaurin, Necula, Paszke, Vander{P}las, Wanderman-{M}ilne, and
  Zhang]{jax2018github}
James Bradbury, Roy Frostig, Peter Hawkins, Matthew~James Johnson, Chris Leary,
  Dougal Maclaurin, George Necula, Adam Paszke, Jake Vander{P}las, Skye
  Wanderman-{M}ilne, and Qiao Zhang.
\newblock {JAX}: composable transformations of {P}ython+{N}um{P}y programs,
  2018.
\newblock URL \url{http://github.com/google/jax}.

\bibitem[Schober et~al.(2014)Schober, Duvenaud, and
  Hennig]{schober2014probabilistic}
Michael Schober, David~K Duvenaud, and Philipp Hennig.
\newblock Probabilistic {ODE} solvers with {Runge--Kutta} means.
\newblock \emph{Advances in Neural Information Processing Systems}, 27, 2014.

\bibitem[Mazzia and Cash(2015)]{mazzia2015fortran}
Francesca Mazzia and Jeff~R Cash.
\newblock A {Fortran} test set for boundary value problem solvers.
\newblock In \emph{AIP Conference Proceedings}, volume 1648. AIP Publishing,
  2015.

\bibitem[Bosch et~al.(2021)Bosch, Hennig, and Tronarp]{bosch2021calibrated}
Nathanael Bosch, Philipp Hennig, and Filip Tronarp.
\newblock Calibrated adaptive probabilistic {ODE} solvers.
\newblock In \emph{International Conference on Artificial Intelligence and
  Statistics}, pages 3466--3474. PMLR, 2021.

\bibitem[Dempster et~al.(1977)Dempster, Laird, and Rubin]{dempster1977maximum}
Arthur~P Dempster, Nan~M Laird, and Donald~B Rubin.
\newblock Maximum likelihood from incomplete data via the {EM} algorithm.
\newblock \emph{Journal of the Royal Statistical Society: Series B
  (Methodological)}, 39\penalty0 (1):\penalty0 1--22, 1977.

\bibitem[Wu(1983)]{wu1983convergence}
CF~Jeff Wu.
\newblock On the convergence properties of the {EM} algorithm.
\newblock \emph{The Annals of Statistics}, pages 95--103, 1983.

\bibitem[Houtekamer and Mitchell(2005)]{houtekamer2005ensemble}
Peter~L Houtekamer and Herschel~L Mitchell.
\newblock Ensemble {Kalman} filtering.
\newblock \emph{Quarterly Journal of the Royal Meteorological Society: A
  Journal of the Atmospheric Sciences, Applied Meteorology and Physical
  Oceanography}, 131\penalty0 (613):\penalty0 3269--3289, 2005.

\bibitem[Gibson and Ninness(2005)]{gibson2005robust}
Stuart Gibson and Brett Ninness.
\newblock Robust maximum-likelihood estimation of multivariable dynamic
  systems.
\newblock \emph{Automatica}, 41\penalty0 (10):\penalty0 1667--1682, 2005.

\bibitem[Seeger(2004)]{seeger2004low}
Matthias Seeger.
\newblock Low rank updates for the {Cholesky} decomposition.
\newblock 2004.

\end{thebibliography}

\appendix

\section{Kalman filter recursions}
\label{appendix-section-kalman-filter}

The {Kalman filter} \citep{kalman1960new} computes $p(x_K \mid y_{1:K})$ by initialising $p(x_0 \mid y_{1:0}) = p(x_0)$ and alternating
\begin{subequations}
\begin{align}
\text{Prediction:}&
&
p(x_{k-1} \mid y_{1:k-1}) 
&\longmapsto p(x_{k} \mid y_{1:k-1})
&
k&=1,...,K
\\
\text{Update:}&
&
p(x_{k} \mid y_{1:k-1}) 
&\longmapsto p(x_{k} \mid y_{1:k})
&
k&=1,...,K.
\end{align}
\end{subequations}
Since all operations are linear and all distributions are Gaussian, the recursions are available in closed form.
Their parametrisation depends on the parametrisation of Gaussian variables as follows.

\subsection{Covariance-based parametrisation}
\label{appendix-section-covariance-based-kalman-filter}

Let $k=1, ..., K$.
The prediction step computes the predicted distribution $p(x_k \mid y_{1:k-1}) = \mathcal{N}(m_{k \mid k-1}, C_{k \mid k-1})$ from the previous filtering distribution $p(x_{k-1} \mid y_{1:k-1})=\mathcal{N}(m_{k-1 \mid k-1}, C_{k-1 \mid k-1})$ by
\begin{align}
m_{k \mid k-1} = A_k m_{k-1 \mid k-1},
\quad
C_{k \mid k-1} = A_k C_{k-1 \mid k-1} (A_k)^\top + B_k.
\end{align}
The update step takes the predicted distribution, forms the joint distribution 
\begin{align}
p(x_k, y_k \mid y_{1:k-1}) 
= p(y_k \mid x_k) p(x_k \mid y_{1:k-1}) 
= \mathcal{N}(H_k x_k, R_k) \mathcal{N}(m_{k \mid k-1}, C_{k \mid k-1})
\end{align}
and computes the update as $p(x_k \mid y_{1:k}) = \mathcal{N}(m_{k \mid k}, C_{k \mid k})$,
\begin{subequations}
\begin{align}
s_{k \mid k-1} &= H_k m_{k \mid k-1},
\\
S_{k \mid k-1} &= H_k C_{k \mid k-1} (H_k)^\top + R_k,
\\
Z_k &= C_{k \mid k-1} (H_k)^\top (S_{k \mid k-1})^{-1},
\\
m_{k \mid k} &= m_{k \mid k-1} + Z_k(y_k - s_{k \mid k-1}),
\\
C_{k \mid k} &= C_{k \mid k-1} - Z_k S_{k \mid k-1} (Z_k)^\top.
\end{align}
\end{subequations}
Iterating these two steps from $k=1, ..., K$ yields $p(x_K \mid y_{1:K})$.

\subsection{Cholesky-based parametrisation}
\label{appendix-section-Cholesky-based-kalman-filter}

The Cholesky-based parametrisation predicts the mean like the covariance-based parametrisation.
The covariance prediction is replaced by the QR decomposition 
\begin{align}
\mathfrak{Q} \mathfrak{R} 
=
\begin{pmatrix}
(\Chol{C_{k-1\mid k-1}})^\top (A_k)^\top
\\
(\Chol{B_k})^\top 
\end{pmatrix}
\end{align}
followed by setting $\Chol{C_{k \mid k-1}} = \mathfrak{R}^\top$ \citep{kramer2024stable,grewal2014kalman}. 
The logic mirrors that in \Cref{algorithm-cholesky-based-implementation}.
The update step in Cholesky-based arithmetic amounts to a QR-decomposition of \citep{gibson2005robust}
\begin{align}
\mathfrak{Q} 
\begin{pmatrix}
\mathfrak{R}_1
&
\mathfrak{R}_2
\\
0
&
\mathfrak{R}_3
\end{pmatrix}
=
\begin{pmatrix}
(\Chol{R_k})^\top & 0 \\
(\Chol{C_{k \mid k-1}})^\top (H_k)^\top
&
(\Chol{C_{k \mid k-1}})^\top
\end{pmatrix}
\end{align}
followed by setting
\begin{align}
\Chol{S_{k \mid k-1}} = (\mathfrak{R}_1)^\top,
\quad
Z_k = ((\mathfrak{R}_1)^{-1} \mathfrak{R}_{2})^\top,
\quad
\Chol{C_{k \mid k}} =(\mathfrak{R}_3)^\top.
\end{align}
Then, use $Z_k$ to evaluate $m_{k \mid k}$ and iterate $k \mapsto k+1$.
If $p(y_k \mid y_{1:k-1})$ is needed, evalaute $s_{k \mid k-1}$ like in the covariance-based implementation.
Choosing the QR decomposition by \citet{gibson2005robust} avoids implementing Gaussian updates via Cholesky downdates \citep[e.g.][]{yaghoobi2022parallel}, which are sometimes numerically unstable \citep{seeger2004low}.
We refer to \citet[Chapter 4]{kramer2024implementing} for why the above assignments yield the correct distributions.

\section{Rauch--Tung--Striebel smoother recursions}
\label{appendix-section-rts-smoother}
The Rauch--Tung--Striebel smoother proceeds similarly to the Kalman filter.

\subsection{Covariance-based parametrisation}
The prediction step in the smoother involves computing $p(x_k \mid y_{1:k-1})$ like in the filter.
It further evaluates
\begin{subequations}
\begin{align}
p(x_{k-1} \mid x_{k}, y_{1:k-1})
&=
\mathcal{N}(G_{k-1 \mid k} x_k + p_{k-1 \mid k}, P_{k-1 \mid k}),
\\
G_{k-1 \mid k} 
&=
C_{k-1 \mid k-1} (A_k)^\top (C_{k \mid k-1})^{-1}
\\
p_{k-1 \mid k}
&=
m_{k-1 \mid k-1} - G_{k-1 \mid k} m_{k \mid k-1}
\\
P_{k-1 \mid k}
&=
P_{k-1 \mid k-1} - G_{k-1 \mid k} C_{k+1 \mid k} (G_{k-1 \mid k})^\top.
\end{align}
\end{subequations}
This conditional distribution is stored after every step. The rest of the step proceeds like in the Kalman filter.

\subsection{Cholesky-based parametrisation}
\label{appendix-section-Cholesky-based-rts-smoother}

The Cholesky-based parametrisation computes the same conditional distribution as the covariance-based parametrisation but replaces matrix multiplication and matrix addition with another QR decomposition;
specifically, the QR decomposition
\citep{gibson2005robust}
\begin{align}
\mathfrak{Q} 
\begin{pmatrix}
\mathfrak{R}_1
&
\mathfrak{R}_2
\\
0
&
\mathfrak{R}_3
\end{pmatrix}
=
\begin{pmatrix}
(\Chol{B_k})^\top & 0 \\
(\Chol{C_{k-1 \mid k-1}})^\top (A_k)^\top
&
(\Chol{C_{k-1 \mid k-1}})^\top
\end{pmatrix}.
\end{align}
This QR decomposition is followed by setting
\begin{align}
\Chol{C_{k \mid k-1}} = (\mathfrak{R}_1)^\top,
\quad
G_{k-1 \mid k} = ((\mathfrak{R}_1)^{-1} \mathfrak{R}_{2})^\top,
\quad
\Chol{P_{k-1 \mid k}} =(\mathfrak{R}_3)^\top.
\end{align}
This step computes the Cholesky factors of $C_{k \mid k-1}$ and  $P_{k-1 \mid k}$, as well as the smoothing gain $G_{k - 1\mid k}$ in a single sweep. It replaces the prediction step of the Cholesky-based Kalman filter.
Compute $m_{k \mid k-1}$ and $p_{k-1 \mid k}$ like in the covariance-based parametrisation.
Store these quantities and proceed with the update step of the Cholesky-based Kalman filter.
Again, refer to \citet[Chapter 4]{kramer2024implementing} for why these assignments yield the correct distributions.

\section{Proof of \Cref{proposition-reduction}}
\label{appendix-section-proof-of-reduction-proposition}

Recall the notation from \Cref{algorithm-covariance-based-implementation}, most importantly, $G_{0 \mid 0}=I$, $p_{0 \mid 0} = 0$, and $P_{0 \mid 0} = 0$.
\Cref{algorithm-fixed-point-smoother} computes the transition 
\begin{align}
p(x_0 \mid x_k, y_{1:k-1})
=
\mathcal{N}( G_{0\mid k} x_k + p_{0\mid k}, P_{0 \mid k} ).
\end{align}
A forward pass with a Kalman filter gives
\begin{align}
p(x_k \mid y_{1:k}) = \mathcal{N}(m_{k \mid k}, C_{k \mid k}).
\end{align}
Then, due to the rules of manipulating Gaussian distributions,
\begin{align}
p(x_0 \mid y_{1:k})
=
\mathcal{N}(G_{0 \mid k} m_{k \mid k} + p_{0 \mid k}, G_{0 \mid k} C_{k \mid k} (G_{0 \mid k})^\top + P_{0 \mid k})
\end{align}
follows.
To show that this matches the result of \Cref{equation-fixed-point-smoother-recursion}, it suffices to show
\begin{subequations}
\label{equation-proof-show}
\begin{align}
p_{0 \mid k} &= m_{0 \mid k-1} - G_{0 \mid k} m_{k \mid k-1} 
\label{equation-proof-show-mean}
\\
P_{0 \mid k} &= C_{0 \mid k-1} - G_{0 \mid k} C_{k \mid k-1} (G_{0 \mid k})^\top
\label{equation-proof-show-cov}
\end{align}
\end{subequations}
because the remaining terms already coincide.
We use induction to show \Cref{equation-proof-show}.

\paragraph{Initialisation}
For $k=1$, we get
\begin{align}
p_{0 \mid 1} = G_{0 \mid 0} p_{0 \mid 1} + p_{0 \mid 0} = p_{0 \mid 1} = m_{0 \mid 0} - G_{0 \mid 1} m_{1 \mid 0},
\end{align}
which shows \Cref{equation-proof-show-mean} (compare the smoothing recursion in \Cref{appendix-section-rts-smoother} for the last step),
as well as 
\begin{align}
P_{0 \mid 1} = G_{0 \mid 0} P_{0 \mid 1} G_{0 \mid 0} + P_{0 \mid 0} = P_{0 \mid 1}
\end{align}
which gives \Cref{equation-proof-show-cov}.

\paragraph{Step}
Now, let \Cref{equation-proof-show} hold for some $k$.
Then, (abbreviate ``RTSS'': ``Rauch--Tung--Striebel smoother'')
\begin{align}
p_{0 \mid k+1} 
&= G_{0 \mid k} p_{k \mid k+1} + p_{0\mid k}
\tag{\Cref{algorithm-covariance-based-implementation}}
\\
&= G_{0 \mid k} p_{k \mid k+1} + m_{0 \mid k-1} - G_{0 \mid k} m_{k \mid k-1}
\tag{\Cref{equation-proof-show-mean} holds for $k$}
\\
&= G_{0 \mid k} ( m_{k\mid k} - G_{k \mid k+1} m_{k+1 \mid k} ) + m_{0 \mid k-1} - G_{0 \mid k} m_{k \mid k-1}
\tag{Expand $p_{k \mid k+1}$ as in RTSS}
\\
&= G_{0 \mid k} m_{k\mid k} -G_{0 \mid k} G_{k \mid k+1} m_{k+1\mid k}  + m_{0 \mid k-1} - G_{0 \mid k} m_{k \mid k-1}
\tag{Resolve parentheses}
\\
&= m_{0 \mid k-1} - G_{0 \mid k}(m_{k \mid k-1} - m_{k \mid k}) -G_{0 \mid k+1} m_{k \mid k+1}
\tag{$G_{0 \mid k} G_{k \mid k+1} = G_{0 \mid k+1}$}
\\
&= m_{0 \mid k} -G_{0 \mid k+1} m_{k+1 \mid k}
\tag{\Cref{equation-fixed-point-smoother-recursion}}
\end{align}
\Cref{equation-proof-show-mean} is complete.
Similarly,
\begin{align}
P_{0 \mid k+1}
&=
G_{0 \mid k} P_{k \mid k+1} (G_{0 \mid k})^\top + P_{0 \mid k}
\tag{\Cref{algorithm-covariance-based-implementation}}
\\
&=
G_{0 \mid k} P_{k \mid k+1} (G_{0 \mid k})^\top + C_{0 \mid k-1} - G_{0 \mid k} C_{k \mid k-1} (G_{0 \mid k})^\top
\tag{\Cref{equation-proof-show-cov} holds for $k$}
\\
&=
G_{0 \mid k} (C_{k \mid k} - G_{k \mid k+1} C_{k+1 \mid k} (G_{k \mid k+1})^\top) (G_{0 \mid k})^\top + C_{0 \mid k-1} - G_{0 \mid k} C_{k \mid k-1} (G_{0 \mid k})^\top
\tag{Expand $P_{k \mid k+1}$ as in RTSS}
\\
&=
G_{0 \mid k} C_{k \mid k} (G_{0 \mid k})^\top - G_{0 \mid k+1} C_{k+1 \mid k} (G_{0 \mid k+1})^\top + C_{0 \mid k-1} - G_{0 \mid k} C_{k \mid k-1} (G_{0 \mid k})^\top
\tag{Expand parentheses, $G_{0 \mid k} G_{k \mid k+1} = G_{0 \mid k+1}$}
\\
&=
C_{0 \mid k-1} - G_{0 \mid k} (C_{k \mid k-1} - C_{k \mid k}) (G_{0 \mid k})^\top - G_{0 \mid k+1} C_{k+1 \mid k} (G_{0 \mid k+1})^\top
\tag{Rearrange}
\\
&= 
C_{0 \mid k} - G_{0 \mid k+1} C_{k+1 \mid k} (G_{0 \mid k+1})^\top
\tag{\Cref{equation-fixed-point-smoother-recursion}}
\end{align}
the covariance recursion must hold.
\Cref{equation-proof-show} holds for all $k \geq 1$, and the statement is complete.

\end{document}